\documentclass{article}
\usepackage{amsmath}
\usepackage{cancel}
\usepackage{stmaryrd}
\input prooftree.sty
\usepackage{proof}
\usepackage{amssymb}
\usepackage{amsthm}
\usepackage{amsmath}
\usepackage{xspace}
\usepackage{url}
\usepackage{cancel}
\usepackage{stmaryrd}
\theoremstyle{plain}
\newtheorem{theorem}[equation]{Theorem}
 \newtheorem{corollary}[equation]{Corollary}
 \newtheorem{lemma}[equation]{Lemma}
 \newtheorem{proposition}[equation]{Proposition}
 \newtheorem{definition}[equation]{Definition}
\theoremstyle{definition}
 \newtheorem{remark}{\emph{Remark}}

\newcommand{\fv}{\mathsf{free}}

\newcommand{\finset}[1]{\{ #1 \}} 
\newcommand{\thehole}[1]{  \finset{  #1  }}

\newcommand{\sq}[2]{#1\vdash #2}

\renewcommand{\implies}{\Rightarrow}

\newcommand{\GS}{\textbf{GS}\xspace}

\newcommand{\We}{\textsc{W}}
\newcommand{\C}{\textsc{C}}

\newcommand{\w}{\wedge}
\newcommand{\G}{\mathrm{\Gamma}}

\newcommand{\Ax}{\textsc{Ax}}

\newcommand{\wR}{\w\mathrm R}

\newcommand{\vR}{\vee\mathrm R}

\newcommand{\WR}{\We\mathrm R}

\newcommand{\CR}{\C\mathrm R}

\newcommand{\existsR}{\exists\mathrm R}

\newcommand{\forallR}{\forall\mathrm R}

\newcommand{\turnstile}{\vdash}

\title{A sequent calculus demonstration of Herbrand's Theorem}
\author{Richard McKinley\thanks{Work supported by the Swiss National Science Foundation grant ``Algebraic and Logical Aspects of Knowledge Processing.''}}

\begin{document}
\maketitle
\begin{abstract}
  Herbrand's theorem is often presented as a corollary of Gentzen's
  sharpened Hauptsatz for the classical sequent calculus.  However, the
  midsequent gives Herbrand's theorem directly only for formulae in
  prenex normal form.  In the Handbook of Proof Theory, Buss claims to
  give a proof of the full statement of the theorem, using sequent
  calculus methods to show completeness of a calculus of
  \emph{Herbrand proofs}, but as we demonstrate there is a flaw in the
  proof.

  In this note we give a correct demonstration of Herbrand's theorem
  in its full generality, as a corollary of the full cut-elimination
  theorem for LK.  The major difficulty is to show that, if there is
  an Herbrand proof of the premiss of a contraction rule, there is an
  Herbrand proof of its conclusion.  We solve this problem by showing
  the admissibility of a \emph{deep contraction rule}.
\end{abstract}
\section{Introduction}

Herbrand's fundamental theorem \cite{Herbrand30} gives that
provability in the predicate calculus may be reduced to propositional
provability: specifically, given any formula $A$ in the language of
first-order logic, we can compute, given a proof of $A$, a valid
quantifier-free formula built from substitution instances of
subformulae of $A$. Herbrand's theorem most easily stated for
$\exists$-formulae or $\forall\exists$-formulae, and this form of the
theorem is sufficient for applications. The most general form of the
theorem that most students of logic will see is for a disjunction of
prenex formulae, as this follows as an almost immediate consequence of
Gentzen's midsequent theorem (or sharpened Hauptsatz)
\cite{Gentzen34}. The following is the midesequent theorem for $\GS$,
a one-sided sequent system with multiplicative context handling, as
shown in Table~\ref{GS1}:

\begin{theorem}
  Suppose that $\Gamma$ is a sequence of prenex formulae provable in
  the system \GS. Then there is some quantifer-free sequent $\Gamma'$
  and a proof in \GS of the form

\[
\deduce[\phantom{N} \  \Arrowvert \ M]{\sq{}{\G}}
{\deduce[\phantom{M} \  \Arrowvert \ N]{\sq{}{\G'}}{}}
\]
where the derivation  \ $\mathrm{M}$ necessarily
contains only propositional rules, and where $N$
contains only introductions of quantifiers and structural rules.  The
sequent $\G'$ is then called the \emph{midsequent}.
\end{theorem}
\noindent From this one may easily extract a form of Herbrand's
theorem for prenex sequents, see for example \cite{TroelSchwi96BasProThe}.
\begin{table}
\small
\begin{center}
\[
\begin{prooftree}
\justifies
\turnstile a, \bar{a}
\using \Ax
\end{prooftree}\]
\[
\begin{prooftree}
\turnstile \G, A, B
\justifies
\turnstile \G, A\lor B
\using \vR
\end{prooftree}
\qquad\qquad
\begin{prooftree}
\turnstile \G, A \qquad \turnstile \G', B
\justifies 
\turnstile \G, \G', A \land B
\using \wR
\end{prooftree} \]
\[
\begin{prooftree}
\turnstile \G, A, A
\justifies
\turnstile \G, A
\using \CR
\end{prooftree}
\qquad\qquad\qquad\qquad
\begin{prooftree}
\turnstile \G
\justifies 
\turnstile \G, A
\using \WR
\end{prooftree}\]
\[
\qquad \qquad \quad \begin{prooftree}
\turnstile \G, A(t)
\justifies
\turnstile \G, \exists y.A
\using \existsR
\end{prooftree}
\qquad\qquad \qquad 
\begin{prooftree}
\turnstile \G, A(z)
\justifies 
\turnstile \G, \forall z.A
\using \forallR \quad \text z \notin  \fv( \G)
\end{prooftree}\]
\caption{System \GS}
\label{GS1}
\end{center}
\end{table}

The original theorem, as stated by Herbrand, was more general, and
stated in terms of a system of proofs for first-order classical logic.
The opening chapter of the Handbook of Proof Theory, by Buss
{\cite{buss98introproof}}, gives a readable presentation of a variant
of this system called ``Herbrand proofs''. The general version of
Herbrand's theorem can be rendered thus: a formula is valid if and
only if it has an Herbrand proof.  Buss gives a proof of this
statement, which relies on the following incorrect lemma: the system
$\GS$ given above is complete when the contraction rule is resticted
to quantifier-free formulae and formulae whose main connective is an
existential quanitifier.  To see that this does not hold, consider the
sequent
\[\sq{}{\forall x.A\w \forall x. B, (\exists x.\bar{A} \vee \exists
   x. \bar{ B})\w(\exists x.\bar{ A} \vee \exists x. \bar{ B})} \]
 \noindent The application of any rule of $\GS$ other than contraction
 on the rightmost conjunction yields an invalid sequent.

 Of course, Herbrand's theorem does hold in the form stated by Buss.
 In this note we give a repaired proof of Herbrand's theorem which,
 like Buss's attempt, derives the theorem from the cut-free
 completeness of $\GS$.  We give a construction yielding, from a
 cut-free $\GS$ proof of a formula $\Gamma$, an Herbrand proof of
 $\G$; thus the general Herbrand's theorem is shown to be a corollary
 of the \emph{general} cut-elimination for the first-order classical
 sequent calculus, rather than of the midesequent theorem.  We prove
 this by showing
 that each rule of $\GS$ is \emph{admissible} in the Herbrand proofs
 system: given an Herbrand proof of the premises one may obtain an
 Herbrand proof of the conclusion.  The only non-trivial case is that
 of contraction, where the admissibility of contraction for formulae
 of rank $<n$ is not enough to demonstrate admissibility of
 contraction on rank $n$ formulae; instead, we show that a more
 general \emph{deep} contraction rule is admissible.

\subsection{Conventions}
Formulae of first-order logic are always written in negation normal
form (that is, negation is primitive only at the level of atoms, with
the negation of a general formula being given by the De Morgan laws).
The \emph{rank} of a formula is its depth as a tree.  We consider
formulae of first-order logic modulo the renaming of bound variables
($\alpha$-equivalence).  A formula $A$ will said to be
\emph{alpha-normal} if there is at most one occurence of a quantifier
$q.x$ in $A$, where $x$ is a variable and $q$ either $\forall$ or
$\exists$.  Every formula $A$ is $\alpha$-equivalent to an alpha-normal
formula.  A formula $B$ is in \emph{prenex normal form} if it has the
form $Q.M$, where $M$ contains no quantifiers and $Q$ is a sequence of
quantifiers.  In that case, we call $M$ the \emph{matrix} of $B$.

We assume a particular form of variable use for sequent proofs in
$\GS$. Variables should be used strictly: each universal rule binds a
unique eigenvariable, and that eigenvariable occurs only in the
subproof above the rule which binds it.  Further, we enforce a
Barendregt-style convention on the use of variables: the sets of bound
and free variables appearing in a proof should be disjoint.

\section{Herbrand proofs}
 
We give first the definition of Herbrand proofs as formulated by Buss
\cite{buss95onherb}.  
\begin{remark}
  We consider, for cleanness of presentation, only pure first-order
  logic over a signature of relation symbols and function symbols,
  containing at least one constant symbol.  Extending our approach to
  one dealing theories containing equality or with nonempty sets of
  nonlogical axioms may be done with no change in the shape of our
  argument.
\end{remark}

We begin with three key definitions:
\begin{definition}
Let $A$ be a formula in negation normal form.  An
\emph{$\lor$-expansion} of $A$ is any formula obtained from $A$ by a
finite number of applications of the following operation:

If $B$ is a subformula of an $\lor$-expansion $A'$ of $A$, replacing
$B$ in $A'$ with $B\lor B$ produces another $\lor$-expansion of $A$.

A \emph{strong $\lor$-expansion} of $A$ is defined similarly, except
that now the formula $B$ is restricted to be a subformula with
outermost connective an existential quantifier.
\end{definition}
(Note that by this definition, $A$ is a (strong-$\lor$) expansion of
itself.) From now on, we will abbreviate ``strong $\lor$-expansion''
to ``expansion''.  An expansion $\hat{\G}$ of a sequent $\G = A_1
\dots A_n$ is a sequence $\hat{A_1} \dots \hat{A_n}$ of expansions of
the members of $\G$.

\begin{definition} Let $A$ be an alpha-normal formula.  A \emph{prenexification} of $A$
  is a formula $B$ in prenex normal form derived from $A$ by successive applications of the
  operations
\[qx.A  * B \rightsquigarrow qx.(A  * B)  \qquad A  * qx.B \rightsquigarrow qx.(A  * B)\] 
\noindent(where $q$ is either $\forall $ or $\exists $, and $*$ is
either $\land$ or $\lor$).  If $\bigvee \G$ is alpha-normal, a
prenexification of $\G$ is a prenexification of $\bigvee
\G$.
\end{definition}

\begin{definition} 
  Let $A$ be a valid alpha-normal first-order formula in prenex normal
  form.  If $A$ contains $r \geq 0$ existential quantifiers, then $A$
  is of the following form, with $B$ quantifier free: \[ (\forall x_1
  \cdots \forall x_{n_1})(\exists y_1)(\forall x_{n_1 + 1} \cdots
  \forall x_{n_2})(\exists y_2) \cdots (\exists y_r)(\forall x_{n_r +
    1} \cdots \forall x_{n_{r+1}}) B(\bar{x}, \bar{y})\] \noindent
  with $0 \leq n_1 \leq n_2 \leq \cdots \leq n_{r+1}$.  A
  \emph{witnessing substitution} for $A$ is a sequence of terms $t_1,
  \dots, t_r$ such that (1) each $t_i$ contains arbitrary free
  variables but only bound variables from $x_1,\dots, x_{n_i}$, and
  (2) the formula $B(\bar{x}, t_1, \dots, t_n)$ is a tautology.

\end{definition}

\noindent We are now ready to define Herbrand proofs:

\begin{definition}[Buss] An \emph{Herbrand proof} of a first-order
  formula $A$ consists of a prenexification $A^*$ of a strong
  $\lor$-expansion of $A$, plus a witnessing substitution $\sigma$ for
  $A^*$.
\end{definition}

We will need the more general notion of an Herbrand proof of a sequent:
\begin{definition} An \emph{Herbrand proof} of a sequent $\G$ is
  a triple consisting of a strong~$\lor$-expansion $\hat{\G}$ of
  $\G$, a prenexification $\G^*$ of  $\hat{\G}$, and a witnessing substitution $\sigma$ for $\G^*$.
\end{definition}
\section{The proof of Herbrand's theorem}

We show next that the system of Herbrand proofs as given above is
complete --- each valid sequent has an Herbrand proof.  We prove that
each rule of the system $\GS$ is \emph{admissible}; that is, whenever
we we have an Herbrand proof or proofs of the premises, we have an
Herbrand proof of the conclusion. Since $\GS$ is complete for
first-order classical logic, this will be enough to show completeness
of the Herbrand proofs system.  Proving admissibility is trivial for
most of the rules of $\GS$, and we leave the proof as an exercise:
\begin{proposition}
\label{prop:admis}
  Let $\rho\in\finset{\Ax, \wR, \vR, \forallR, \existsR}$.  Then,
  for any instance of $\rho$, if there is are Herbrand proofs of the
  premisses, there is an Herbrand proof of the conclusion.
\end{proposition}

The admissibility of weakening relies on the presence of a constant in
the signature over which we work:
\begin{proposition}
\label{prop:weak}
Let $A$ be a formula of first-order logic.  Then if $\G$ has an Herbrand
proof, so does $\G, A$.  
\end{proposition}
\begin{proof}
  Let $(\hat{\G}, Q.C, \sigma)$ be an Herbrand proof of $\G$.  Let
  $Q'.D$ be a prenexification of $A$ sharing no bound variables with
  $Q.C$.  Then we form an Herbrand proof
  \[((\hat{\G}, A),\  Q'.Q.(C\lor D), \  \sigma') \]
  of $\G, A$, where $\sigma'$ assigns the same term as $\sigma$ to 
  existential quantifiers in $Q$, and assigns a constant term
  $\mathsf{c}$ to all existentially bound variables in $Q'$.
\end{proof}

The only rule to pose some difficulty is contraction. We would like to
prove contraction admissible by induction on the rank of a formula to
be contracted, but this induction hypothesis is not strong enough.  To
see this, suppose that we have shown contraction admissible for all
formulae of rank $\leq n$, and let $A \land B$ have rank $n+1$.  Given
an Herbrand proof
\[((\hat{\Gamma}, \hat{A}_1\land\hat{B}_1, \hat{A}_2\land\hat{B}_2), \ 
\G^*, \  \sigma)\] 
of $\Gamma, A \land B, A \land B$, how do we use our induction
hypothesis to produce a proof of $\Gamma, A \land B$ ? We can get
close by using the valid implication
\begin{equation}
\label{eq:medial}
 (A \land B) \lor (C \lor D) \implies (A \lor C) \land (B \lor D). 
\end{equation}
\begin{remark}
  This implication plays an important role in the proof-theoretic
  formalism known as \emph{deep inference}, where it is known as
  \emph{medial}.  It is used as an inference rule in Br\"unnler's system SKS
  \cite{BruLCL06} to reduce contraction to atomic form.  Its use here
  is similar.
\end{remark}

\begin{lemma}
\label{lem:medial}
If 
\[((\hat{\Gamma}, \hat{A}_1\land\hat{B}_1, \hat{A}_2\land\hat{B}_2), \ 
Q.C, \  \sigma)\] 
is an Herbrand proof of $\Gamma, A \land B, A \land B$, then 
\[((\hat{\Gamma}, (\hat{A}_1\lor \hat{A}_2)\land
(\hat{B}_1\lor\hat{B}_2)), \ Q.C', \ \sigma)\] is an Herbrand proof of
$\Gamma, (A \lor A)\land(B \lor B)$, where $Q.C'$ is the unique (up to
associativity of $\lor$) prenexification of $\hat{\Gamma},
(\hat{A}_1\lor \hat{A}_2)\land (\hat{B}_1\lor\hat{B}_2)$ with
quantifier prefix $Q$.
\end{lemma}
\begin{proof}
  It is clear that $\hat{\Gamma}, (\hat{A}_1\lor \hat{A}_2)\land
  (\hat{B}_1\lor\hat{B}_2)$ is an expansion of \mbox{$\Gamma, (A \lor
  A)\land(B \lor B)$}, and has a prenexification of the form $ Q.C'$.
  We must check that $\sigma$ is a witnessing substitution for $
  Q.C'$.  Since it is a witnessing substitution for $Q.C$, it
  satisfies the condition on free variables of substituting terms, and
  we need only check that $\sigma(C')$ is a tautology.  Let $A_i^*$ be
  the matrix of $\hat{A}_i$, $B_i^*$ be the matrix of $\hat{B}_i$, and
  $G$ be the matrix of $\hat{\G}$.  Then we know, since $\sigma$ is
  a witnessing substitution for $Q.C$, that
  \[ \sigma(G) \lor (\sigma(A^*_1)\land
  \sigma(B^*_1))\lor(\sigma(A^*_2)\land \sigma(B^*_2)) \]
  is a tautology.  Applying \eqref{eq:medial}, we conclude that
   \[ \sigma(C') =\sigma(G) \lor (\sigma(A^*_1)\lor
  \sigma(A^*_2))\land(\sigma(B^*_1)\lor \sigma(B^*_2)) \]
  is a tautology.
\end{proof}

Of course in the sequent calculus one can only apply contractions
across a comma, so even in this case we may not apply our induction
hypothesis.  To move forward we will need to show admissibility of a
``deep'' contraction rule, which can act on arbitrary subformulae in
the ednsequent.  Admissibility of ordinary, ``shallow'', contraction
follows immediately.  We will need the following definitions:
\begin{definition}
\begin{enumerate}
\item A \emph{one-hole-context} is a sequent with precisely 
one positive occurrence of the special atom $\thehole{}$ (the hole).  We write
$\Gamma \thehole{}$ to denote a one hole context.
\item An \emph{$n$-hole-context} is a sequent with precisely one
  positive occurrence each of the $n$ special atoms $\thehole{}_1
  \dots \thehole{}_n$.  We write $\Gamma \thehole{} \dots \thehole{}$
  to denote an $n$ hole context, where by convention $\thehole{}_1$ is
  the leftmost hole in the sequent etc..
\item If $\G\thehole{}$ is a one hole context, we write $\Gamma\thehole{A}$
for the sequent given by replacing the hole with $A$.  Similarly for
$n$ hole contexts.
\end{enumerate}
\end{definition} 
The following easy lemma will be crucial.
\begin{lemma}
An expansion of a sequent $\Gamma\thehole{A}$ has the form
$\hat{\Gamma} \thehole{A_1} \dots \thehole{A_n}$, where $A_1 \dots A_n$
are expansions of $A$ and $\hat{\Gamma} \thehole{} \dots \thehole{}$
is an expansion of $\Gamma \thehole{}$
\end{lemma}

\begin{lemma}
The deep contraction rule
\[\begin{prooftree}
\G\thehole{A \lor A}
\justifies
\G\thehole{A}
\using 
\mathrm{Deep}\C
\end{prooftree}\]
\noindent is admissible for Herbrand proofs.
\end{lemma}
\begin{proof}
By induction on the structure of $A$:
\begin{itemize}
\item Suppose we have
an Herbrand~proof 
\[(\hat{\Gamma}\thehole{a\lor a}\dots\thehole{a\lor a}, \ Q.C, \
\sigma)\]
of $\Gamma\thehole{a\lor a}$.  Then clearly there is an Herbrand proof
\[(\hat{\Gamma}\thehole{a}\dots\thehole{a}),\ Q.C',\ \sigma)\] of
$\Gamma\thehole{a}$.

Now suppose, for each remaining case, that deep contraction is admissible
for formulae of rank $\leq n$, and that $A$ has rank $n+1$.
\item  Suppose $A=B \lor C$, and that we have an Herbrand proof
\[(\hat{\Gamma}\thehole{(\hat{B}_{11}\lor
  \hat{C}_{11})\lor(\hat{B}_{12}\lor
  \hat{C}_{12})}\dots\thehole{(\hat{B}_{n1}\lor
  \hat{C}_{n1})\lor(\hat{B}_{n2}\lor \hat{C}_{n2})}, \ Q.C, \ \sigma)\] of
$\Gamma\thehole{(C\lor D)\lor( C\lor D)}$. Then 
\[(\hat{\Gamma}\thehole{(\hat{B}_{11}\lor
  \hat{B}_{12})\lor(\hat{C}_{11}\lor
  \hat{C}_{12})}\dots\thehole{(\hat{B}_{n1}\lor
  \hat{B}_{n2})\lor(\hat{C}_{n1}\lor \hat{C}_{n2})}\] is an expansion
of of $\Gamma\thehole{(C\lor C)\lor(D \lor D)}$, with a
prenexification $Q.C'$; these two, plus $\sigma$, give us an Herbrand
proof of $\Gamma\thehole{(C\lor C)\lor(D \lor D)}$.  Apply the
induction hypothesis to obtain an Herbrand proof of
$\Gamma\thehole{C\lor D}$
\item 
Suppose $A= \exists x.B$, with $B$ of rank $n$.  An Herbrand
proof of \newline $\Gamma\thehole{\exists x.B\lor\exists y.B}$ has the form
\[(\hat{\Gamma}\thehole{\hat{A}_{11}\lor
  \hat{A}_{12}}\dots\thehole{\hat{A}_{m1}\lor \hat{A}_{m2}}, \ Q.C,
\ \sigma).\] 
But this is also an Herbrand proof of
$\Gamma\thehole{\exists x.B}$, since if $A_1$ and $A_2$ are expansions
of an existential formula $\exists x.B$, then so is $A_1 \lor A_2$.
\item
Suppose $A= \forall x.B$, with $B$ of rank $n$.  
 An Herbrand
proof of \newline $\Gamma\thehole{\forall x.B\lor \forall y.B}$ has the form
\[(\hat{\Gamma}\thehole{\forall x_1 \hat{B}_{11}\lor \forall
  y_1. \hat{B}_{12}}\dots\thehole{\forall x_m \hat{B}_{m1}\lor \forall
  y_m. \hat{B_{m2}}}, \ Q.C, \ \sigma).\] Suppose we are given such a
proof. We generate a new sequence $Q'$ of quantifiers as follows: let
$z_i$ stand for the first occurrence a member of $\finset{x_i, y_i}$
in $Q$ and $w_i$ for the second.  Let $Q'$ be the result of deleting
each occurence of $\forall w_i$ from $Q$ (so that, from each pair
$\finset{\forall x_i, \forall y_i}$, we keep the first and discard the
second.)  $Q'$ contains the same existential variables $x_1, \dots,
x_r$ as $Q$, and in the same order.  Let $\sigma = t_1, \dots t_r$.
Let $t'_i = t_i [w_1:=z_1]\dots [w_m:=z_m]$, and let $\sigma' = t'_1,
\dots t'_r$.

The following is then  an Herbrand proof of $\Gamma\thehole{\forall z.(B\lor B)}$:
\[(\hat{\Gamma}\thehole{\forall z_1 ((\hat{B}_{11}\lor  \hat{B}_{12})
  [w_1:=z_1])}\dots\thehole{\forall z_m ((\hat{B}_{m1}\lor
  \hat{B}_{m2}) [w_m:=z_m])}, Q'.C', \ \sigma'),\]  
By the induction hypothesis, we derive an Herbrand proof of $\Gamma\thehole{\forall z.B}$.

\item Finally, suppose that $A = B\land C$. Then, if we have a
  Herbrand proof of $\Gamma\thehole{(B\land C)\lor(B\land C)}$, it
  consists of an expansion of the form
\[\hat{\Gamma}\thehole{(\hat{B}_{11}\land
    \hat{C}_{11})\lor(\hat{B}_{12}\land
    \hat{C}_{12})}\dots\thehole{(\hat{B}_{m1}\land
    \hat{C}_{m1})\lor(\hat{B}_{m2}\land \hat{C}_{m2})}\]
a prenexification $Q.C$ and a substitution $\sigma$.  The formula 
\[\hat{\Gamma}\thehole{(\hat{B}_{11}\lor
  \hat{B}_{12})\land(\hat{C}_{11}\lor
  \hat{C}_{12})}\dots\thehole{(\hat{B}_{m1}\lor
  \hat{B}_{m2})\land(\hat{C}_{m1}\lor \hat{C}_{m2})}\] is an expansion
of $\Gamma\thehole{(B\lor B)\land(C\lor C)}$, and it can be easily
seen that it has a prenexification of the form $Q.C'$.  By the same
reasoning used to prove Lemma~\ref{lem:medial}, $\sigma(C')$ is a
tautology, and therefore $\sigma$ is a witnessing substitution for
$Q.C'$.  This gives an Herbrand~proof of $\Gamma\thehole{(B\lor
  B)\land(C\lor C)}$.  Apply the induction hypothesis twice to obtain
an Herbrand proof of $\Gamma\thehole{B\land C}$
\end{itemize}
\end{proof}

\begin{corollary}
  \label{cor:contrad}
Contraction is admissible for Herbrand proofs.
\end{corollary}

\begin{theorem}
A formula of first-order logic is valid if and only if it has 
an Herbrand~proof.
\end{theorem}
\begin{proof}
  Follows immediately from Propositions~\ref{prop:admis}
  and~\ref{prop:weak}, Corollary~\ref{cor:contrad} and the cut-free
  completeness of $\GS$.
\end{proof}

\section{Conclusions}
As we have seen, Herbrand's theorem in its full generality can be
seen as a consequence of cut-elimination for the sequent calculus (and
not, as usually claimed, of the midsequent theorem).  To show this, we
had to consider an extended sequent calculus with a deep contraction
rule, and show that each proof in that extended calculus gives rise to
an Herbrand proof.  This raises some potentially interesting
questions: are there Herbrand proofs which arise from a proof with
deep contraction, but not from any shallow proof?  If so, is there an
easy condition separating the ``shallow'' Herbrand proofs from the
``deep''?

For the special case of the prenex Herbrand theorem, the author has
studied in \cite{Mck10Herbnets} the elimination of \emph{cuts} in
Herbrand's theorem, giving a notion of Herbrand proofs with cut and
showing a syntactic cut-elimination theorem.  This works because of a
strong connection bewteen the structure of prenex Herbrand proofs and
the corresponding midesequent-factored sequent proofs.  The Herbrand
proofs we present in this paper have a similar strong connection to
proofs with deep contraction.  It is unlikely that we can find a
similar cut-elimination result for general Herbrand proofs without a
cut-elimination result for the system with deep contractions.
Syntactic cut elimination for the system with deep contraction
seems to be a very challenging problem.

This note began with the observation that restricting the contraction
rule to existential and quantifier-free formulae broke completeness.
The crucial observation is that contraction on $A\land B$ does not
follow inductively from contraction on $A$ and contraction on $B$.
Instead of moving to deep contraction, we can instead simply add back
contraction for conjunctions (so now we only disallow contraction on
disjunctions and universal quantifications).  This gives rise to an
Herbrand-like theorem in which first-order provability is reduced to
provability in a well-behave fragment of multiplicative linear logic.
This is ongoing work.

\bibliography{nets}
\end{document}